\documentclass{amsart}
\usepackage{algorithm}
\usepackage{algorithmic}
\usepackage{amsmath,amsthm}
\usepackage{verbatim}
\usepackage{marvosym}
\usepackage{graphicx}
\usepackage{color}
\usepackage{epsfig}
\usepackage{rotating}
\usepackage{lscape}
\usepackage{bm}
\usepackage{mathrsfs, euscript, mathdesign}
\usepackage[normalem]{ulem}
\usepackage{cleveref}
    \usepackage{baskervald}
\usepackage[T1]{fontenc}

\usepackage{tikz}
\usetikzlibrary{cd}
\usepackage{thmtools}

\theoremstyle{plain}
\newtheorem{thm}{Theorem}[section]
\newtheorem{cor}[thm]{Corollary}

\newtheorem{prop}[thm]{Proposition}
\newtheorem{lemma}[thm]{Lemma}

\theoremstyle{definition}

\newtheorem{question}[thm]{Question}

\makeatletter
\newcommand*{\house}[1]{%
  \mathord{%
    \mathpalette\@house{#1}%
  }%
}
\newcommand*{\@house}[2]{%
  \dimen@=\fontdimen8 %
      \ifx#1\scriptscriptstyle\scriptscriptfont
      \else\ifx#1\scriptstyle\scriptfont
      \else\textfont\fi\fi
      3 %
  \sbox0{%
    $#1%
      \vrule width\dimen@\relax
      \overline{%
        \kern2\dimen@
        \begingroup 
          #2%
        \endgroup
        \kern2\dimen@
      }%
      \vrule width\dimen@\relax
      \mathsurround=1.5\dimen@ 
    $%
  }%
  \ht0=\dimexpr\ht0-\dimen@\relax
  \dp0=\dimexpr\dp0+2\dimen@\relax
  \vbox{%
    \kern\dimen@ 
    \copy0 
  }%
}

\newcommand{\Bx}{\ensuremath{\mathcal{B}}}
\newcommand{\Ox}{\ensuremath{\mathcal{O}}}

\newcommand{\Q}{\ensuremath{\mathbb{Q}}}
\newcommand{\R}{\ensuremath{\mathbb{R}}}
\newcommand{\Z}{\ensuremath{\mathbb{Z}}}
\newcommand{\C}{\ensuremath{\mathbb{C}}}

\renewcommand{\Re}{\text{Re}}
\renewcommand{\Im}{\text{Im}}

\newcounter{nootje}
\begin{document}
\title[Mahler Measure in Cubic Fields]{Minimal Mahler Measure in Cubic Number Fields}

\author[L. Eldgedge and K. Petersen ]{Lydia Eldredge and Kathleen L. Petersen }

\thanks{The second author was supported by Simons Foundation Grant 430077}



\begin{abstract}
The minimal integral Mahler measure of a number field $K$, $M(\Ox_K)$, is the minimal Mahler measure of an integral generator of $K$.  Upper and lower bounds, which depend on the discriminant and degree of $K$, are known.  We show that for three natural families of cubics, the  lower bounds are sharp with respect to its growth as a function of discriminant. We construct an algorithm to compute $M(\Ox_K)$ for all cubics with absolute value of the discriminant bounded by $N$ and show the resulting data for $N=10,000$.

\end{abstract}
 
\maketitle


\section{Introduction}

The Mahler measure of  a non-constant  polynomial $f(x)=c \prod_{i=1}^d (x-\alpha_i)\in \C[x]$ is  
\[
M(f) = |c| \prod_{ |\alpha_i|\geq 1} |\alpha_i|
\]
and for an algebraic number we define $M(\alpha)$ to be the Mahler measure of a minimal polynomial for $\alpha$ over $\Q$ (with content 1).  
We define the minimal Mahler measure of a number field $K$ to be the minimal Mahler measure of a generator,
\[
M(K) = \min\{ M(\alpha): \Q(\alpha) = K \}
\]
and we write $M(\Ox_K)$ for the minimal Mahler measure of an integral generator. 
The minimal Mahler measure has been linked to bounds for $\ell$-torsion in class groups \cite{MR3778549} and systoles of arithmetic  surfaces \cite{MR3747173}. 
Let $D_K$ denote the absolute discriminant of a number field $K$. In this work, we study the dependence of  $M(\Ox_K)$ on  $|D_K|$ for cubic number fields $K$. 

Silverman \cite{MR747871}  showed that for all $\alpha$ of degree $d\geq 2$, 
 \[
d^{-\frac{d}{2(d-1)}}|D_{\Q(\alpha)}|^{\frac1{2(d-1)}}\leq  M(\alpha).
 \]
Our first result, proven in Section~\ref{section:lowerbounds},  shows that this bound  cannot be improved (by a larger exponent on the discriminant) for cubic fields of various types.
\begin{thm}\label{thm:lowerbound}
There are infinitely many Galois cubics, totally real non-Galois cubics, and non-totally real cubics $K$ such that 
\[  M(\Ox_K) \leq 2^{\frac12} |D_K|^{\frac14}. \]
\end{thm}
In Section~\ref{section:1/3example} we prove that there are cubic fields  whose minimal integral Mahler measure is related to the discriminant to a power different from $\tfrac14$. 
\begin{thm}\label{thm:exponent1/3}
There are infinitely many  non-Galois (Kummerian) cubic number fields $K$ such that 
\[ \tfrac1{30} |D_K|^{\frac13}<  M(\Ox_K)<  \tfrac{4}{3} |D_K|^{\frac13}.  \]
\end{thm}
In Section~\ref{section:algorithm} we describe an algorithm that given a discriminant bound, $N$ computes $M(\Ox_K)$ for all number fields $K$ with absolute value of their discriminant less than $N$, and present  the resulting data for $N=10,000$ in Figure~\ref{figure:data} .

\subsection{Background}

Silverman's result \cite{MR747871} can be stated as the lower bound bound  
 \[ d^{-\frac{d}{2(d-1)}}|D_K|^{\frac1{2(d-1)}}\leq  M(K)  \]
 for $d=[K:\Q]\geq 2$.  Ruppert \cite{MR1624340} (page 18) and Masser \cite{MR1363143} (Proposition 1) provided a family of fields for which this lower bound  cannot be improved by a larger exponent on the discriminant. (See also \cite{MR3413883}.)  They considered the Kummerian fields $K$ of degree $d$ obtained by adjoining a root of $px^d-q$ for $p$ and $q$ certain  well-chosen  primes.  For these fields    $M(K)\leq \sqrt{2} |D_K|^{\frac1{2(d-1)}}$. 
Vaaler and Widmer  \cite{MR3413883} showed that  for composite $d$ there are number fields $K$ for which no constant $c_d$ satisfies
\[
M(K)  \leq c_d |D_K|^{\frac1{2(d-1)}}, 
\]
demonstrating that there are fields whose minimal Mahler measure grows faster than Silverman's lower bound.

Upper bounds, some of which are dependent on the truth of the generalized Riemann hypothesis (GRH), show that $M(\Ox_K)$ is bounded above by a constant times $|D_K|^{\frac12}$. 
Ruppert \cite{MR1624340}(Proposition 3)  proved 
 that if $K$ is totally real of prime degree then there is an integral primitive element $\alpha\in \Ox_K$ such that 
\[
 M(\alpha)\leq    |D_K|^{\frac12}.
\] 
(Ruppert's proof showed that the naive height of $\alpha$ is bounded above by $2^d  |D_K|^{\frac12}$ and can be modified to demonstrate the given bound.)
His argument  can be extended to all number fields of prime degree using an extension of Minkowski's linear forms theorem.
Vaaler and Widmer  \cite{MR3074815}  proved that if $K$ is not totally complex, and $r_2$ denotes the number of complex places of $K$,  then there is a primitive element $\alpha$ such that 
\[
M(\alpha) \leq \Big(\frac{2}{\pi}\Big)^{r_2} |D_K|^{\frac1{2}}.
\]
For general fields they show that an upper bound with the same dependence on $D_K$ exists under the assumption of the truth of the generalized Riemann hypothesis.

For quadratic number fields, the upper and lower bounds have the same exponent (equal to one half) on the  discriminant and $\tfrac12 \sqrt{|D_K|} \leq M(K) \leq \sqrt{|D_K|}$.  Cochrane et al \cite{MR3463562} showed this for real quadratics using elementary techniques.  They also  computed $M(K)$ where $K=\Q(\sqrt{a})$  for all positive, square-free $a$  up to one million and conjectured that 
\[ \lim_{a \rightarrow \infty} \frac{M(K)}{\sqrt{|D_K|}} =\frac12\]
and proved that the associated limit inferior is $\tfrac12$.

Our work considers the integral Mahler measure in the  cubic case, where the aforementioned bounds specialize to 
\[ 3^{-\frac{3}{4}}|D_K|^{\frac1{4}}\leq  M(\Ox_K) \leq  |D_K|^{\frac12}. \]
Theorem~\ref{thm:lowerbound} shows that the exponent $\tfrac14$ on the discriminant is sharp in the lower bound, whereas  Theorem~\ref{thm:exponent1/3} shows that there are infinitely many fields where $M(\Ox_K)$ is related to $|D_K|$ by an intermediate exponent,  $\tfrac13$.  Figure~\ref{figure:data} shows the results of our algorithm.  We used a supercomputer to compute $M(\Ox_K)$ for all cubic number fields $K$ with absolute value of their discriminant less than $10,000$.   Our results  leave open the question if the upper bound is achieved.   
\begin{question} Is there a constant $c>0$ such that for   infinitely many  cubic fields $K$,  $M(\Ox_{K})>c|D_{K}|^{\tfrac12}$? \end{question}

\section{Lower Bound Examples}\label{section:lowerbounds}

In this section we prove Theorem~\ref{thm:lowerbound} by demonstrating  infinitely many number fields $K$ in each of the three categories of Galois cubics, totally real non-Galois cubics, and non-totally real cubics, which have $M(\Ox_K)$ equal to a constant less than $\sqrt{2}$  times $|D_K|^{\frac14}$.

\begin{prop}
There are infinitely many cyclic cubic number fields $K$ such that \[ M(\Ox_K)< 2^{\frac12} |D_K|^{\frac14}.\]
\end{prop}

\begin{proof}

Let $K_{n}$ be the splitting field of $f_{n}(x)=x^3+nx^2-(n+3)x+1$.  These fields have been extensively studied and are often referred to as ``the simplest cubic fields''.  (See, for example  \cite{MR352049} with  the substitution of $-x$ for $x$ in the minimal polynomial.) 
The field  $K_n$ is a Galois cubic for $n\geq 0$ and  the field discriminant equals the discriminant of the polynomial $f_n$.  Specifically, we have
\[ D_{K_n}=\text{disc}(f_n)=(n^2+3n+9)^2. \] 
An elementary computation using the intermediate value theorem shows that for $n$ sufficiently large
 the three real  roots $r_1<r_2<r_3$ of $f_n$ satisfy 
\[ -(n+1+\tfrac2n)  <r_1< -(n+1), \quad 0<r_2<\tfrac{1}{n}, \quad 1<r_3<1+\tfrac{1}n.\]
Therefore, 
\[ 
M(f_n)< (n+1+\tfrac{2}n)(1+\tfrac{1}{n}) =n+2+\tfrac3n+\tfrac{2}{n^2}
\]
and so for $n$ sufficiently large,  we have 
\[
M(f_n)^2 <  n^2+ 4n + 10    + \tfrac{16}n + \tfrac{17}{n^2} + \tfrac{12}{n^3}+ \tfrac{4}{n^4} 
<
2 (n^2+3n+9) = 2 |D_K|^{\tfrac12}.
\]

\end{proof}

Next we consider totally real, non-Galois cubics. 
\begin{prop}\label{prop:totrealnongal}
There are infinitely many totally real non-Galois cubic number fields $K$ such that 
\[ M(\Ox_K) < |D_K|^{\frac14}.\]
\end{prop}

\begin{proof}
Let $g_n(x) = x^3-nx^2+n$ and $K_n=\Q(\alpha_n)$ for a (preferred) root $\alpha_n$ of $g_n$.  The discriminant of $g_n$ is $\text{disc}(g_n) = n^2(4n^2-27)$. For $n>2$ the discriminant is positive and therefore the field discriminant, $D_{K_n}$,  is positive and $K_n$ is totally real.  If $n$ is square-free, then $g_n$ is Eisenstein  for any prime $p$ dividing $n$.  Such a prime $p$ is totally ramified in $K$ and divides $D_{K_n}$, so if $n$ is square-free then $n$ divides $D_{K_n}$.

  By  \cite{MR1512732} there are infinitely many $n$ such that $4n^2-27$ is  square-free. Ricci  \cite{ricci} later determined an asymptotic formula for the number of square-free solutions.  This can be generalized to prime values of $n$ (see \cite{MR3327840}), and we conclude that there are infinitely many prime values of  $n$ such that $4n^2-27$ is square-free.  
For these $n$, by the above discussion,  the   $\text{disc}(g_n)$ must equal $D_{K_n}$, as $\text{disc}(g_n)$ equals a square times $D_{K_n}$.   The resulting field $K_n$ is a non-Galois totally real field with discriminant $D_{K_n}=n^2(4n^2-27)$.

By the intermediate value theorem,  for $n>6$ the polynomial  $g_n$ has real roots $r_1<r_2<r_3$ satisfying 
\[
-1<r_1<-1+\tfrac1n, \quad 1<r_2<1+\tfrac1n, \quad n-\tfrac2n<r_3<n.
\]
We conclude that 
\[ M(\Ox_{K_n}) < (1+\tfrac1n) n =n+1.  \]
An elementary  calculation shows that for $n$ sufficiently large
\[
M(\Ox_K)^4 < (n+1)^4 = n^4+4n^3+6n^2+4n + 1 <  4n^4-27n^2 = D_{K_n}. 
\]

\end{proof}

Finally, we consider    cubics with a complex place. 

\begin{prop}
There are infinitely many non-totally real   cubic number fields $K$ such that 
\[ M(\Ox_K) < 2^{-\frac12} |D_K|^{\frac14}.\]
\end{prop}

\begin{proof}
Let $h_n(x) = x^3+nx^2+n$ and $K_n=\Q(\alpha_n)$ for a (preferred) root $\alpha_n$ of $h_n$.  The discriminant of $h_n$ is $\text{disc}(h_n) = -n^2(4n^2+27)$ and therefore $K_n$ has negative discriminant and  is not totally real.  If $n$ is square-free, then $h_n$ is Eisenstein for any prime $p$ dividing $n$. Such a prime is totally ramified in $K$ and  divides $D_{K_n}$, so if $n$ is square-free then $n$ divides $D_{K_n}$.  As in  the proof of Proposition~\ref{prop:totrealnongal}, there are infinitely many prime values of $n$ such that $4n^2+27$ is square-free.  For these $n$ the discriminant $\text{disc}(h_n)$ must equal $D_{K_n}$.  The resulting field $K_n$ is a non-Galois non-totally real field with discriminant $D_{K_n}= -n^2(4n^2+27)$. 

By the intermediate value theorem, for $n>0$  the real root $r$ of $h_n$ satisfies 
\[
-n-\tfrac1{n} < r< -n.
\]
Let $\tau$ and $\bar{\tau}$ be the non-real roots of $h_n$.  Using elementary symmetric functions,  
\[
-n = r\tau \bar{\tau}=r|\tau|^2
\]
and since $|r|\geq |n|$ we conclude that $|\tau|^2<1$.

Therefore $M(h_n)=|r|$ since $\tau$ and $\bar{\tau}$ have modulus less than one.  We have $M(h_n) = |r| < |n +\tfrac{1}{n}|$
 and 
\[
M(h_n)^4 <(n+\tfrac1n)^4 < \tfrac14n^2(4n^2+27)= \tfrac14 |D_{K_n}| 
\]
where the second inequality holds for $n>1$.

\end{proof}

In the above proofs, one can also make a similar deductions, with a less optimal constant by considering the naive height of the minimal polynomial.

\section{Exponent of $1/3$ Example}\label{section:1/3example}

In this section we prove Theorem~\ref{thm:exponent1/3}.  First we collect some information about Minkowski embeddings.

\subsection{Minkowski Embedding}

Let $K$ be a number field of degree $d=r_1+2r_2$ where $r_1$ is the number of real places and $r_2$ is the number of complex places of $K$.   Let $\phi_1,\dots, \phi_{r_1}$ be the real places of $K$ and $\tau_1,\dots, \tau_{r_2}$ be representatives of the complex places of $K$.  
We define the Minkowski embedding   $\varphi:  K\rightarrow \R^{r_1+r_2}$  as
\[
\varphi(\alpha) = \langle \phi_1(\alpha), \dots, \phi_{r_1}(\alpha), \Re(\tau_1(\alpha)), \Im(\tau_1(\alpha)), \dots,\Re(\tau_{r_2}(\alpha)), \Im(\tau_{r_2}(\alpha))    \rangle.
\]
Let $\| \cdot \|$ denote the (vector)  length  in $\R^{r_1+2r_2}$ and let ${\bf 1}=\varphi(1)$.

\begin{lemma}\label{lemma:mmlength}
Let $K$ be a number field and $\alpha \in \Ox_K$. Then 
\[
M(\alpha)^{1/d} \leq \|  \varphi(\alpha) \| \leq \sqrt{r_1+r_2} \: M(\alpha).
\]
\end{lemma}

\begin{proof}
Let $\alpha_1, \dots, \alpha_d$ be the conjugates of $\alpha$. 
We have that 
\begin{align*}
 \| \varphi(\alpha)\|^2 
 & =  \sum_{i=1}^{r_1} | \phi_i(\alpha) |^2 + \sum_{j=1}^{r_2} \Big( |\Re(\tau_j(\alpha))|^2 +  | \Im(\tau_j(\alpha))|^2 \Big)  \\
& = \sum_{i=1}^{r_1} | \phi_i(\alpha) |^2 + \sum_{j=1}^{r_2}   | \tau_j(\alpha) |^2.
\end{align*}
Therefore, 
\begin{align*}
 \| \varphi(\alpha)\|^2 
&  \leq (r_1+r_2) \max \{  | \phi_i(\alpha) |^2 , | \tau_j(\alpha) |^2 \}  \\
& \leq  (r_1+r_2) ( \max_{i} |\alpha_i|)^2  \leq (r_1+r_2) M(\alpha)^2.
\end{align*}
This proves the right inequality.  
To prove the left inequality, since $M(\alpha)\leq \max_i |\alpha_i|^d$, we have 
\[
M(\alpha)^{1/d}    \leq  \max_{i} |\alpha_i|   \leq   \Big(\sum_{i=1}^d|\alpha_i|^2 \Big)^{\frac12}=  \| \varphi(\alpha)\|. \]

\end{proof}

\subsection{Proof of Theorem~\ref{thm:exponent1/3} }

Consider the polynomial  $f(x)=x^3-p$ for $p$ prime.
As such, $f$ has roots $\theta$, $\omega \theta$ and $\omega^2 \theta$ for $\theta\in \R$ and $\omega = \tfrac12(-1+\sqrt{-3})$  a primitive third root of unity.  Let $K=\Q(\theta)$.
Assume that $p \not \equiv \pm 1 \pmod 9$ so that $D_K=-27p^2$ and $\{1,\theta,\theta^2\}$ is an integral basis for $\Ox_K$.  (See \cite{MR0457396} page 28, for example.) Further, let $k$ be an integer satisfying  $-\tfrac12 < \alpha <\tfrac12$ for  $\alpha = \theta-k$.   
Theorem~\ref{thm:exponent1/3} follows from Lemma~\ref{lemma:exponent1} and Lemma~\ref{lemma:exponent2} below.

\begin{lemma}\label{lemma:exponent1}
The Mahler measure of $\alpha$ is at most $\tfrac43 |D_K|^{\frac13}$.
\end{lemma}
\begin{proof}
The conjugates of $\alpha$ are $\alpha=\theta-k$, $\theta \omega-k$ and $\theta \omega^2-k=\theta \bar{\omega}-k$. Since $|\alpha|<1$ the Mahler measure of $\alpha$  is 
\[
| \theta \omega-k|^2 =(\theta \omega-k)(\theta \bar{\omega}-k) = \theta^2 \omega^3 -k\theta(\omega+\bar{\omega})+k^2 = 
\theta^2+k^2+k\theta. \]
Since $0<k<\theta+\tfrac12$ this is bounded above by $3\theta^2+\tfrac32\theta+\tfrac14$ which is less that $4\theta^2$ for $p\geq 5$. Therefore, 
\[ M(\alpha) = | \theta \omega-k|^2 <4\theta^2 =4 p^{\frac23}= \tfrac43 |D_K|^{\frac13}.
\] 
\end{proof}

Theorem~\ref{thm:exponent1/3} follows from the next lemma. 

\begin{lemma}\label{lemma:exponent2}
There is no $\beta \in \Ox_K-\Z$ with   $M(\beta) < \tfrac1{30} |D_K|^{\frac13}.$ 
\end{lemma}

\begin{proof}

We will assume by way of contradiction that  $M(\beta)<  \tfrac1{10} p^{\frac23}=  \tfrac1{10}\theta^2= \tfrac1{30} |D_K|^{\frac13}$. 
First, we will show that under this assumption if $\beta=a+b\theta+c\theta^2$ with $a,b,c\in \Z$ then $c=0$.
To this end, we consider the Minkowski embedding, $\varphi$, of $\Ox_K$ induced by sending basis elements $1, \theta$ and $\theta^2$ to vectors in $\mathbb{R}^3$ as follows:
\begin{align*}
1 & \mapsto {\mathbf v_1}= {\bf 1}=\langle 1, 1,0 \rangle \\
\theta  & \mapsto {\mathbf v_2}=\langle \theta, \Re(\omega \theta), \Im(\omega \theta) \rangle  = \langle \theta, -\tfrac12 \theta,  \tfrac{\sqrt{3}}2 \theta \rangle = \tfrac{\theta}2 \langle 2, -1 ,  \sqrt{3}  \rangle  \\
\theta^2 & \mapsto {\mathbf v_3}=\langle \theta^2, \Re(\omega^2 \theta^2), \Im(\omega^2 \theta^2) \rangle= \tfrac{\theta^2}2 \langle 2, -1 ,  -\sqrt{3}  \rangle.
\end{align*}
By Lemma~\ref{lemma:mmlength} the length of the Minkowski embedding is related to Mahler measure by  \[ M(\beta)^{\frac13}\leq \| \varphi(\beta) \| \leq \sqrt{2} M(\beta).\]
Our bounded Mahler measure assumption implies that  
\[ \| \varphi(\beta)\|^2 \leq 2 M(\beta)^2<  \tfrac{2}{100}p^{\frac43}=\tfrac{1}{50} \theta^4.\]

We perform the Gram-Schmidt algorithm to determine an orthogonal basis consisting of ${\mathbf v_1^*}= \mathbf{ v_1}$, $\mathbf{v_2^*}$ and $\mathbf{v_3^*}$ for $K$.  Explicitly, 
\[ \mathbf{v_2^*} = \mathbf{v_2 } - \frac{\mathbf{v_2} \cdot\mathbf{ v_1^*}}{\| \mathbf{v_1^*} \|^2}\mathbf{v_1^*}, \quad \mathbf{v_3^*} = \mathbf{v_3}-\frac{\mathbf{v_3} \cdot \mathbf{v_1^*}}{\| \mathbf{ v_1^*} \|^2}\mathbf{v_1^*} -\frac{ \mathbf{v_3}\cdot \mathbf{v_2^*}}{\|\mathbf{v_2^*}\|^2} \mathbf{v_2^*}. \]
An elementary calculation shows that 
\[ \mathbf{v_2^*}= \mathbf{v_2}- \tfrac14 \theta \mathbf{v_1^*}   = \tfrac14 \theta \langle3, -3  ,   2\sqrt{3} \rangle  \]
and 
\[ 
\mathbf{v_3^*}  =  \mathbf{v_3}- \tfrac14 \theta^2 \mathbf{v_1^*} - \tfrac15 \theta \mathbf{ v_2^*}    = \tfrac35\theta^2 \langle  1,-1,-\sqrt{3}  \rangle.  
\] 
The lengths of the Gram-Schmidt basis elements are
\[ 
\| \mathbf{ v_1^*} \|^2 = 2, \| \mathbf{ v_2^*} \|^2 = \tfrac{15}8 \theta^2, \| \mathbf{v_3^*}\|^2 =   \tfrac{9}{5}\theta^4. 
\]

We   write $\beta$ in terms of the integral basis $\{1,\theta, \theta^2\}$ as $\beta = a+b\theta + c \theta^2$ with $a,b,c\in \Z$ and so that  $b$ and $c$ are not simultaneously zero.  We relate the Minkowski embedding of the integral basis to the Gram-Schmidt basis as follows:  
\begin{align*}
\varphi(\beta) & = a\varphi(1) + b \varphi(\theta) + c\varphi(\theta^2)  \\
& = a \mathbf{v_1}+ b \mathbf{v_2}+c \mathbf{v_3}  \\
& = a \mathbf{v_1^*} + b( \mathbf{v_2^*}+\tfrac14\theta  \mathbf{v_1^*}) + c ( \mathbf{v_3^*}+\tfrac14 \theta^2  \mathbf{v_1^* }+ \tfrac15 \theta \mathbf{v_2^*}) \\
& = (a+\tfrac{b}4\theta +\tfrac{c}4 \theta^2  ) \mathbf{v_1^*} + (b+\tfrac{c}5\theta   ) \mathbf{v_2^*}+c\mathbf{v_3^*} \\
& = A  \mathbf{v_1^*}+B\mathbf{v_2^*}+C\mathbf{v_3^*}
\end{align*}with $A,B,C\in \R$ and $C=c\in \Z$.
Since $\{ \mathbf{v_1^*},  \mathbf{v_2^*},  \mathbf{v_3^*}\}$ is a Gram-Schmidt basis 
\[
\| \varphi(\beta) \|^2 = A^2 \| \mathbf{v_1^*} \|^2 + B^2 \| \mathbf{v_2^*}\|^2 + C^2 \| \mathbf{v_2^*}\|^2 = 2A^2 + \tfrac{15}8 \theta^2 B^2+ \tfrac{9}5 \theta^4 C^2.
\]
We conclude that $|C|<1$ since our assumption implies that $\|\varphi(\beta)\|^2 <  \tfrac{1}{50} \theta^4$ and therefore we must have $\|\varphi(\beta)\|^2 <  \tfrac{9}5 \theta^4$. Therefore since $C=c\in \Z$ we must have $c=0$ and 
 $\beta = a+b\theta$ with $b\neq 0$.

It suffices to consider  $\beta = a+b\theta$ with  $a\geq 0$, $b\neq 0$ under the assumption that $M(\beta)< \tfrac1{10}\theta^2 = \tfrac{1}{10}p^{\frac23}$.  
If $a=0$ then $\beta=b\theta$ and $M(\beta)= |b^3\theta^3 |= |b|^3 p \not < \tfrac{1}{10}p^{\frac23}.$ Therefore we may assume  $a>0$ and $b\neq 0.$
The   minimal polynomial of $\beta$  is $x^3-3ax^2+3a^2x-(a^3+b^3p)$.   
 The conjugates of $\beta$ are 
\[
\beta = a+b\theta, a+\omega b \theta,  \overline{a+ \omega  b \theta}= a+\bar{\omega}b\theta
\]
so that  $M(\beta)=|a+b\theta|$, $|a+\omega b\theta|^2$ or $|a+b\theta | |a+\omega b\theta|^2$.

First, consider the case when $M(\beta)= |a+b\theta|$. This  occurs when   $| a+\omega b\theta|^2<1$, or equivalently when  \[ | a-\tfrac12bp^{\frac13} +\tfrac{\sqrt{-3}}2 bp^{\frac13}|^2<1.\] 
We conclude that  the imaginary part, $\tfrac{\sqrt{3}}2 bp^{\frac13}$  is between $-1$ and $1$ which necessitates that $b=0$. This cannot occur as  $\beta \not\in \Z$. 

It now suffices to consider the remaining two cases, so that  $M(\beta) \geq |a+\omega b\theta|^2$. If $b<0$ then as $a>0$, since $b\in \mathbb Z$,  we have 
\[ M(\beta)\geq |a+\omega b\theta|^2=a^2-ab\theta +b^2 \theta^2 >b^2 \theta^2 >  \tfrac1{10} \theta^2 \]
which contradicts our bounded Mahler measure assumption.  Therefore $b>0$.  Consequently,  $a+b\theta>1$ and $M(\beta)$ is the absolute value of the product of all roots, which is the  absolute value of the constant term in the minimal polynomial for $a+b\theta$.  As such, since $a$ is a  non-negative integer and $b$ is a positive integer, 
\[ M(\beta) = a^3+b^3p= a^3+b^3\theta^3>\tfrac1{10}\theta^2\]
again contradicting our bounded Mahler measure assumption and  completing the proof.

\end{proof}

\section{Algorithm}\label{section:algorithm}

In this section we present our algorithm to compute $M(\Ox_K)$ for a all cubic fields $K$ with bounded $|D_K|$. The data is shown in Figure~\ref{figure:data}. First we prove some lemmas that  establish the  finite search bounds used in the algorithm.

\begin{figure}[h!]\label{figure:data}
\begin{center}
\includegraphics[scale=0.37]{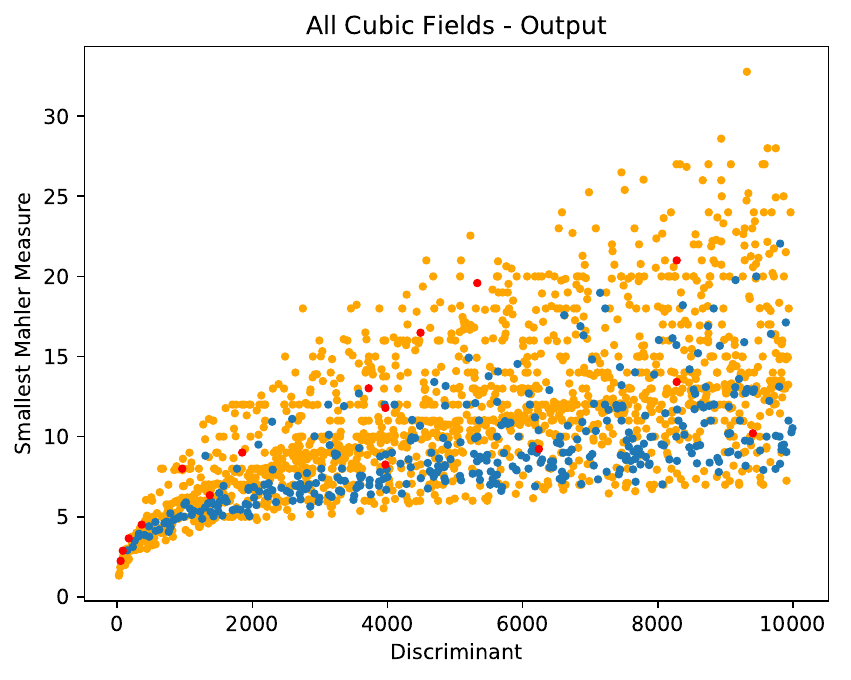} 
\includegraphics[scale=0.37]{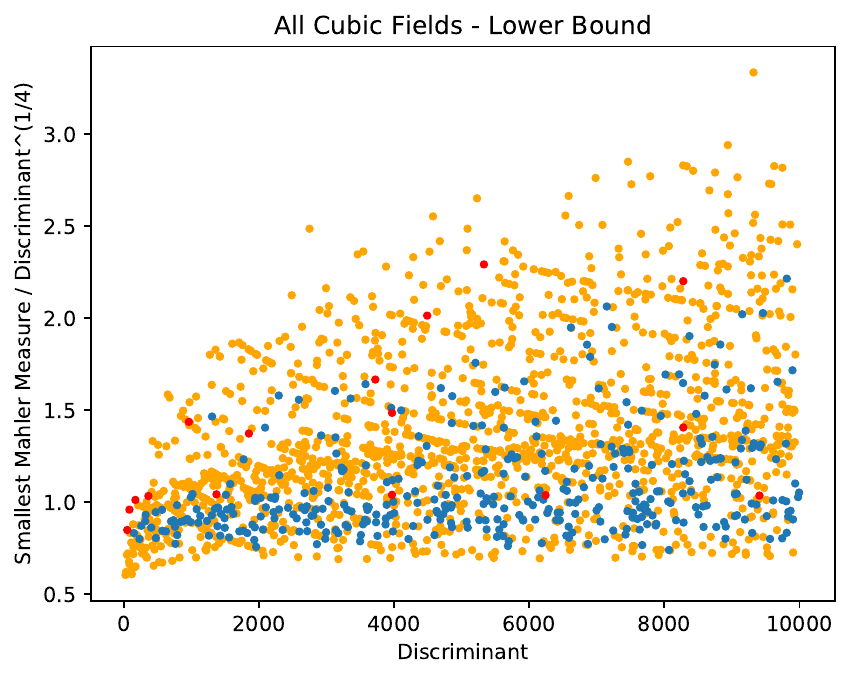} \\

\

\includegraphics[scale=0.37]{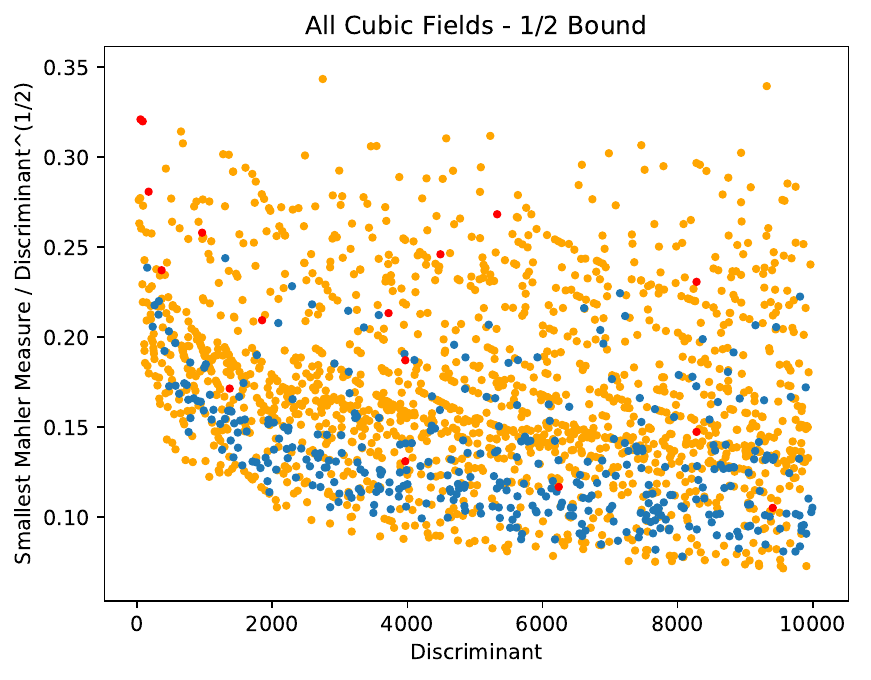} 
\caption{For $K$ cubic, the figures are  $M(\Ox_K)$, $ |D_K|^{-\frac14}M(\Ox_K)$  and $ |D_K|^{-\frac12}M(\Ox_K)$ (respectively) for $K$ with $|D_K|\leq 10,000$.  
Blue indicates non-real fields, orange is non-cyclic totally real fields, and red is cyclic fields.   }

\end{center}
\end{figure}

\subsection{Lemmas}
Let $K$ be a cubic number field with Minkowski embedding $\varphi$.  As such,  $\varphi(K)$ is dense in $\R^3$ and $\varphi(\Ox_K)$ is a lattice in $\R^3$.  The vector $\mathbf{1} = \varphi(1)$ is  given by 
$\mathbf{1} = \langle 1,1,1\rangle$ when $K$ is totally real and $\mathbf{1}=\langle 1,1,0\rangle$ when $K$ has a complex place.  

Let $\Bx= \{ \bm{\beta}_1, \bm{\beta}_2, \bm{\beta}_3\}$  be an LLL basis for $\varphi(\Ox_K)$ and let  $\Bx^*= \{ \bm{\beta}_1^*, \bm{\beta}_2^*, \bm{\beta}_3^*\}$ be the basis obtained by performing Gram-Schmidt reduction on $\Bx$.

\begin{lemma}\label{lemma:searchbounds}
Let $\bm{\alpha}$ be an element of the lattice $\varphi(\Ox_K)$ and $a,b,c\in \Z$ so that $\bm{\alpha} = a \bm{\beta}_1+b \bm{\beta}_2+c \bm{\beta}_3$.
If  $\| \bm{\alpha} \| \leq C$ then 
\[
|a| \leq   C  \Big(  \frac{1}{\| \bm{\beta}_1^*\|} +  \frac{\tfrac12}{\| \bm{\beta}_2^*\|} + \frac{ \tfrac34 }{\| \bm{\beta}_3^*\|} \Big) , \ \ \ 
|b| \leq  C  \Big( \frac{1}{\| \bm{\beta}_2^*\|} +    \frac{\tfrac12}{\| \bm{\beta}_3^*\|}\Big),  \ \ \
|c| \leq  C  \Big( \frac{1}{\| \bm{\beta}_3^*\|}\Big).
\]
\end{lemma}

\begin{proof}
We write $\bm{\alpha} = a_* \bm{\beta}_1^* + b_* \bm{\beta}_2^* + c_* \bm{\beta}_3^*$ for $a_*, b_*, c_* \in \R$ and since $\Bx^*$ is orthogonal we have
\[
\| \bm{\alpha} \|^2 = a_*^2 \| \bm{\beta}_1^*\|^2 + b_*^2 \| \bm{\beta}_2^*\|^2+ c_*^2 \| \bm{\beta}_3^*\|^2.
\]
Since $\| \bm{\alpha} \|^2 \leq C^2$ we have 
\[
|a_*| \leq \frac{C }{ \| \bm{\beta}_1^*\|},  \ \ \ |b_*| \leq \frac{C }{ \| \bm{\beta}_2^*\|}, \ \ \ |c_*| \leq \frac{C }{ \| \bm{\beta}_3^*\|}. 
\]

The Gram-Schmidt algorithm determines $\Bx^*$ from $\Bx$ as follows, where   $\mu_{i,j} = (\bm{\beta}_i\cdot \bm{\beta}_j^*)/\| \bm{\beta}_j^*\|^2$,
\[
\bm{\beta}_i^* = \bm{\beta}_i-\sum_{j=1}^{i-1} \mu_{i,j} \bm{\beta}_j^*
\]
such that 
\[
\bm{\beta}_1^* = \bm{\beta}_1, \ \ 
\bm{\beta}_2^* = \bm{\beta}_2-\mu_{2,1}\bm{\beta}_1,  \ \ 
\bm{\beta}_3^* = \bm{\beta}_3-\mu_{3,1}\bm{\beta}_1-\mu_{3,2}\mu_{2,1}\bm{\beta}_1.
\]
Therefore, we have 
\[
\bm{\alpha} = (a_*-b_*\mu_{2,1} -c_*\mu_{3,2}+c_*\mu_{3,2}\mu_{2,1})\bm{\beta}_1 
+(b_*-c_*\mu_{3,2})\bm{\beta}_2
+(c_*)\bm{\beta}_3
\]
with
\[
a= a_*-b_*\mu_{2,1} -c_*\mu_{3,2}+c_*\mu_{3,2}\mu_{2,1}, 
b=b_*-c_*\mu_{3,2}, 
c=c_*.
\]
This immediately implies the bound for $|c|$.

Since $\Bx$ is an LLL basis, we have that $|\mu_{i,j}| < \tfrac12$.  Therefore, 
\[
|b| = |b_*-c_*\mu_{3,2}| \leq |b_*|+|c_*||\mu_{3,2}| \leq |b_*|+\tfrac12|c_*|
\]
and 
\begin{align*}
|a| & = |a_*-b_*\mu_{2,1} -c_*\mu_{3,2}+c_*\mu_{3,2}\mu_{2,1}| \\
& \leq |a_*|+|b_*||\mu_{2,1}| +|c_*| |\mu_{3,2}| +|c_*||\mu_{3,2}\mu_{2,1}| \\
& \leq |a_*|+\tfrac12 |b_*|+\tfrac12|c_*| +\tfrac14 |c_*|.
\end{align*}
The result follows.

\end{proof}

By Lemma~\ref{lemma:mmlength} we have the following corollary.
\begin{cor}\label{cor:searchbounds} 

Let $\bm{\alpha}$ be an element of the lattice $\varphi(\Ox_K)$ and $a,b,c\in \Z$ so that $\bm{\alpha} = a \bm{\beta}_1+b \bm{\beta}_2+c \bm{\beta}_3$. If $\bm{\beta}_1=\bm{1}$ and $M(\alpha) \leq C$ then 
\[
|a|  \leq   f^{\frac12}C \Big(   \frac{1}{\| \bm{\beta}_1^*\| }\ +  \frac{ \tfrac12 }{\| \bm{\beta}_2^*\| }  +   \frac{ \tfrac34  }{\| \bm{\beta}_3^*\|}\Big),
|b|  \leq f^{\frac12} C \Big(  \frac{1 }{\| \bm{\beta}_2^*\| } +   \frac{ \tfrac12 }{\| \bm{\beta}_3^*\|}\Big), 
|c|  \leq f^{\frac12} C\Big(  \frac{ 1}{\| \bm{\beta}_3^*\| }\Big), 
\]
where $f=3$ if $K$ is totally real and $f=2$ otherwise. 
\end{cor}

\subsection{Overview}

\begin{algorithm}
\caption{Determine $M(\Ox_K)$ for all cubic number fields with $|D_K|\leq N$}
\begin{algorithmic}
\STATE Run Belabas' algorithm on  $N$
\STATE \quad This returns polynomials in the form $p(x)=ax^3 + bx^2 + cx + d$.
\STATE Determine roots $x_1, x_2, x_3$ of $p(x)$
\STATE Use the  integral basis $\{1, ax_1, ax_1^2 + bx_1\}$ for $\Ox_K$.   
\STATE Compute Minkowski embeddings of $1, ax_1, ax_1^2+ bx_1$ to create integral basis $B$ of the lattice.
\STATE Multiply the elements of $B$ by $10^{10}$ and take their floor, creating $B'$.
\STATE Run the LLL algorithm on $B'$ to get an LLL  basis $\mathcal{B}'$ for the lattice.
\STATE Multiply the elements of  $\mathcal{B}'$ by $10^{-10}$ to create $\mathcal{B}$.
\STATE Determine Mahler measure of elements of $\mathcal{B}$.
\STATE Implement Gram-Schmidt on $\mathcal{B}$ to produce orthogonal basis $\mathcal{B}^*$.  
\STATE Determine bounds for coefficients (of elements in $\Ox_K$ written in terms of LLL basis) from Corollary~\ref{cor:searchbounds}, using $\mathcal{B}^*$, and $M(\mathcal{B})$. 
\STATE Compute Mahler measure of the (finite) list of candidate elements. 
\STATE Return the smallest Mahler measure.
\end{algorithmic}
\end{algorithm}

We use  Belabas' algorithm, which  outputs a list of all cubic number  fields with (absolute value of the) discriminant between 0 and $N$ (positive and negative)  along with a polynomial $p(x)=ax^3 + bx^2 + cx + d$ that both generates the field and has the same discriminant as the field.  An integral basis $\{1, ax_1, ax_1^2 + bx_1\}$ can be read off of this data, where $x_1$ is a real root of $p(x)$.  

If $N$ is positive then the  field is  totally real and has 3 real places. If $N$ is negative then the  field is not totally real,  and has one real place and one complex place. 
SageMath \cite{sage} computes the roots of each polynomial, $x_1, x_2, x_3$ with $x_1\in \mathbb{R}$. Using the integral basis, a basis for a Minkowski embedding of $\Ox_K$ has the form 
\[
\langle 1, 1, 1 \rangle, \langle a x_1, ax_2, a x_3 \rangle, \langle a x_1^2+b x_1, a x_2^2+b x_2, a x_3^2+b x_3 \rangle 
\]
when $K$ is totally real and 
\[
\langle 1, 1, 0 \rangle, \langle a x_1,  \Re (a x_2), \Im (a x_2) \rangle, 
\langle a x_1^2+b x_1, \Re (a x_2^2+b x_2), \Im (a x_2^2+b x_2) \rangle 
\]
when $K$ is not totally real. In the latter case, we take $x_2$ to be the complex root which has  a negative  imaginary part.

We run the LLL algorithm as implemented in SageMath \cite{sage}.  This algorithm accepts only integer vectors.  As a result, we multiply our entries by $10^M$ and then truncate them.  After LLL is run, we reverse this process, by multiplying entries by $10^{-M}$.  We also compute the  Gram-Schmidt matrix normalization of this LLL basis.   In practice we use $M=10$, which is large enough not to introduce significant rounding errors \cite{MR4157999}.

We employ a check to verify that the first vector is  $\bm{1}$, which is the case for all fields we calculated.  This simplifies our implementation as $\|\bm{1}\|^2=2$ or $3$. 
The algorithm computes the Mahler measure of the other two basis elements, and records the smaller value.  We use   Corollary~\ref{cor:searchbounds} with this value to  determine  bounds  for coefficients.  Specifically, elements with Mahler measure smaller than the above value must have integer coefficients (in terms of the LLL basis of their Minkowski embedding) satisfying these bounds.   
The algorithm  makes a list of Minkowski embeddings with length bounded by Corollary~\ref{cor:searchbounds}, 
and computes the Mahler measure for each element, returning the smallest Mahler measure and corresponding elements.

\bibliographystyle{amsplain}
\bibliography{mybib.bib}

\noindent\rule{4cm}{.5pt}
\vspace{.25cm}

\noindent {\small Lydia Eldredge}\\
{\small Department of Mathematics, Rose-Hulman Institute of Technology, Terre Haute IN,47803} \\
{\small email: {\tt eldredge@rose-hulman.edu}}

\vspace{.25cm}

\noindent {\small Kathleen Petersen}\\
{\small Department of Mathematics and Statistics, University of Minnesota Duluth, Duluth MN 55812} \\
{\small email: {\tt kpete@umn.edu}}

\end{document}